\documentclass[11pt, reqno]{amsart}
\usepackage{amssymb,amsthm, amsmath, amsfonts}
\usepackage{graphics}
\usepackage{hyperref}
\usepackage[all]{xy}
\usepackage{enumerate}
\usepackage{mathrsfs}

\theoremstyle{plain}
\newtheorem{theorem}{Theorem}

\newtheorem{proposition}[theorem]{Proposition}
\newtheorem{corollary}[theorem]{Corollary}

\theoremstyle{definition}
\newtheorem{definition}[theorem]{Definition}
\newtheorem{remark}[theorem]{Remark}
\newtheorem{notation}[theorem]{Notation}

\DeclareMathOperator{\Hom}{Hom}
\DeclareMathOperator{\id}{id}
\newcommand{\C}{{\mathrm{C}}}
\newcommand{\Mod}{{\mathrm{\mbox{-}Mod}}}
\newcommand{\Ob}{{\mathrm{Ob}}}
\newcommand{\FB}{{\mathrm{FB}}}
\newcommand{\FI}{{\mathrm{FI}}}
\newcommand{\kk}{{\Bbbk}}
\newcommand{\sst}{{\sqcup\{\star\}}}
\newcommand{\ws}{{\widetilde{{S}}}}

\title{On the negative-one shift functor for FI-modules}

\author{Wee Liang Gan}
\address{Department of Mathematics, University of California, Riverside, CA 92521, USA}
\email{wlgan@math.ucr.edu}

\begin{document}

\begin{abstract}
We show that the negative-one shift functor $\ws_{-1}$ on the category of FI-modules is a left adjoint of the shift functor $S$ and a right adjoint of the derivative functor $D$. We show that for any FI-module $V$, the coinduction $QV$ of $V$ is an extension of $V$ by $\ws_{-1}V$.
\end{abstract}

\maketitle

\section{Introduction}

The shift functor $S$ and the derivative functor $D$ on the category of $\FI$-modules have played an essential role in several recent works, for example, \cite{CE, Ga, Li3, LRa, LYu, Ra}. The modest goal of this article is to explain how they are related to the negative-one shift functor $\ws_{-1}$ introduced in \cite{CEFN}. We also explain how the coinduction functor $Q$ defined in \cite{GL} is related to $\ws_{-1}$.

Let us begin by recalling some definitions.

Let $\kk$ be a commutative ring. Let $\C$ be a small category. A \emph{$\C$-module} is a functor from $\C$ to the category of $\kk$-modules. A \emph{homomorphism of $\C$-modules} is a natural transformation of functors. For any $\C$-module $V$ and $X\in\Ob(\C)$, we write $V_X$ for $V(X)$.

Let $\FI$ be the category whose objects are the finite sets and whose morphisms are the injective maps. 

Fix a one-element set $\{\star\}$ and define a functor $\sigma: \FI \to \FI$ by $X \mapsto X \sst$ for each finite set $X$. If $f:X\to Y$ is a morphism in $\FI$, then $\sigma(f): X\sst \to Y \sst$ is the map $f \sqcup \id_{\{\star\}}$. Following \cite[Definition 2.8]{CEFN}, the \emph{shift functor} $S$ from the category of $\FI$-modules to itself is defined by $SV = V \circ \sigma$ for every $\FI$-module $V$.

Suppose $V$ is an $\FI$-module. For any finite set $X$, one has $(SV)_X = V_{X\sst}$. There is a natural $\FI$-module homomorphism $\iota: V \to SV$ whose components $V_X \to (SV)_X$ are defined by the inclusion maps $X \hookrightarrow X\sst$. We denote by $DV$ the cokernel of $\iota: V \to SV$. Following \cite[Definition 3.2]{CE}, we call the functor $D: V \mapsto DV$ the \emph{derivative functor} on the category of $\FI$-modules.

For any finite set $X$, let
\begin{equation*}
(\ws_{-1}V)_X := \bigoplus_{x\in X}V_{X\setminus\{x\}}.
\end{equation*}
Suppose $f: X\to Y$ is an injective map between finite sets. For any $x\in X$, the map $f$ restricts to an injective map $f\mid_{X\setminus \{x\}} : X\setminus \{x\} \to Y\setminus \{f(x)\}$. We let
\begin{equation*}
f_* : (\ws_{-1}V)_X \longrightarrow (\ws_{-1}V)_Y 
\end{equation*}
by the $\kk$-linear map whose restriction to the direct summand $V_{X\setminus \{x\}}$ is the map 
\begin{equation*}
(f\mid_{X\setminus\{x\}})_* : V_{X\setminus\{x\}} \longrightarrow V_{Y\setminus\{f(x)\}}.
\end{equation*}
This defines an $\FI$-module $\ws_{-1}V$. We call the functor $\ws_{-1}: V \mapsto \ws_{-1}V$ the \emph{negative-one shift functor} on the category of $\FI$-modules. (In \cite[Definition 2.19]{CEFN}, a negative shift functor $\ws_{-a}$ is defined for each integer $a\geqslant 1$.)

After recalling some generalities in the next section, we shall show that the negative-one shift functor $\ws_{-1}$ is a left adjoint functor to the shift functor $S$ and a right adjoint functor to the derivative functor $D$; we shall also show that for any FI-module $V$, the coinduction $QV$ of $V$ is an extension of $V$ by $\ws_{-1}V$. These properties seem to have gone unnoticed; for instance, the paper \cite{LRa} constructed a left adjoint functor of $S$ using the category algebra of $\FI$ and called it the induction functor.

\subsection*{Notations}
Suppose $\C$ is a small category. For any $X, Y \in \Ob(\C)$, we write $\C(X,Y)$ for the set of morphisms from $X$ to $Y$. 

We write $\C\Mod$ for the category of $\C$-modules. Suppose $V, W\in \C\Mod$. We write $\Hom_{\C\Mod}(V,W)$ for the $\kk$-module of all $\C$-module homomorphisms from $V$ to $W$. If $\phi \in \Hom_{\C\Mod}(V,W)$, we write $\phi_X : V_X\to W_X$ for the component of $\phi$ at $X\in \Ob(\C)$.

\section{Generalities on adjoint functors}

Let $\C$ and $\C'$ be small categories, and $F: \C\Mod \to \C'\Mod$ a functor. 

For any $X\in \Ob(\C)$, we define the $\C$-module $M(X)$ by
\begin{equation*}
M(X)_Y := \kk \C(X, Y) \quad \mbox{ for each } Y\in \Ob(\C),
\end{equation*}
that is, $M(X)$ is the composition of the functor $\C(X,-)$ followed by the free $\kk$-module functor.

For any morphism $f\in \C(X,Y)$, we have a $\C$-module homomorphism
\begin{equation*}
\rho_f : M(Y) \longrightarrow M(X), \quad g \mapsto gf,
\end{equation*}
where $g\in \C(Y,Z)$ for any $Z\in \Ob(\C)$.
Hence, we obtain a $\C'$-module homomorphism
\begin{equation*}
F(\rho_f) : F(M(Y)) \longrightarrow F(M(X)).
\end{equation*}

\begin{definition} \label{right adjoint}
Define a functor $F^\dag: \C'\Mod \to \C\Mod$ by
\begin{equation*}
(F^\dag (V))_X := \Hom_{\C'\Mod} ( F(M(X)), V ) \quad \mbox{ for each } V\in \C'\Mod, \; X\in \Ob(\C).
\end{equation*}
For any morphism $f\in \C(X,Y)$, the map $f_*: (F^\dag (V))_X \to (F^\dag (V))_Y$ is defined by
\begin{equation*}
f_*(\phi) := \phi \circ F(\rho_f) \quad \mbox{ for each } \phi\in (F^\dag (V))_X.
\end{equation*}
\end{definition}

\begin{proposition} \label{adjoint functor theorem}
Let $\C$ and $\C'$ be small categories. Let $F: \C\Mod \to \C'\Mod$ be a right exact functor which transforms direct sums to direct sums. Then $F^\dag$ is a right adjoint functor of $F$, that is, there is a natural isomorphism
\begin{equation*}
\Hom_{\C'\Mod}(F(V), W) \cong \Hom_{\C\Mod}(V, F^\dag(W)),
\end{equation*}
where $V\in \C\Mod$ and $W\in \C'\Mod$.
\end{proposition}
\begin{proof}
This is a well-known result for modules over rings with identity, see \cite[Theorem 5.51]{Ro}. For a proof in our setting, see \cite[Proposition 1.3 and Theorem 2.1]{PN}.
\end{proof}

The proofs of the main results of this article are essentially exercises in applications of Proposition \ref{adjoint functor theorem}; nevertheless, they do not appear to be obvious.

\begin{remark}
In all our applications of Proposition \ref{adjoint functor theorem}, we have a pair of functors $F,G:\C\Mod\to\C\Mod$ with the following properties. For each $X\in \Ob(\C)$, there exists $X_i \in \Ob(\C)$ for $i\in I_X$ (where $I_X$ is a finite indexing set) such that there is a $\C$-module isomorphism
\begin{equation*}
\eta_X : \bigoplus_{i\in I_X} M(X_i) \longrightarrow F(M(X)).
\end{equation*}
Moreover, for each $V\in\C\Mod$, there is a decomposition $G(V)_X = \bigoplus_{i\in I_X} V_{X_i}$ such that for any $\psi\in \Hom_{\C\Mod}(V,W)$, one has $G(\psi)_X = \bigoplus_{i\in I_X} \psi_{X_i}$. It follows then that there is a $\kk$-module isomorphism 
\begin{equation*}
\alpha_X: (F^\dag (V))_X \longrightarrow G(V)_X, \quad \phi \mapsto \sum_{i\in I_X} \phi_{X_i} ( (\eta_X)_{X_i} (\id_{X_i}) ),
\end{equation*}
where $\id_{X_i}\in M(X_i)_{X_i}$. Therefore, if $F$ is right exact and sends direct sums to direct sums, then by Proposition \ref{adjoint functor theorem}, to show that $G$ is a right adjoint functor of $F$, it suffices to verify that the collection of $\kk$-module isomorphisms $\alpha_X$ for $X\in \Ob(\C)$ are compatible with the $\C$-module structures on $F^\dag (V)$ and $G(V)$.
\end{remark}

\section{Left adjoint of the shift functor}

Let $X$ be a finite set. By \cite[Lemma 2.17]{CEFN}, there is an isomorphism 
\begin{equation} \label{wsM(X)}
\eta_X: M(X\sqcup \{\star\})  \longrightarrow  \ws_{-1}M(X) 
\end{equation}
defined as follows. For any finite set $Y$ and injective map $f:X\sqcup\{\star\}\to Y$, one has $f\mid_X : X\to Y\setminus \{f(\star)\}$. Let $\eta_X(f)$ be the element $f\mid_X$ of the direct summand $M(X)_{Y\setminus \{ f(\star) \}}$ of $(\ws_{-1}M(X))_Y$.

\begin{theorem}
The functor $\ws_{-1}$ is a left adjoint of the shift functor $S: \FI\Mod \to \FI\Mod$.
\end{theorem}

\begin{proof}
The functor $\ws_{-1}$ is exact and transforms direct sums to direct sums. By Proposition \ref{adjoint functor theorem}, we need to show that the functors $\ws_{-1}^\dag$ and $S$ are isomorphic.

Let $V$ be an $\FI$-module. Let $X$ be any finite set. From Definition \ref{right adjoint} and (\ref{wsM(X)}), we have a $\kk$-module isomorphism
\begin{equation*}
\alpha_X: (\ws_{-1}^\dag V)_X \longrightarrow (SV)_X, \quad \phi \mapsto \phi_{X\sqcup\{\star\}}((\eta_X)_{X\sqcup\{\star\}}(\id_{X\sqcup\{\star\}})).
\end{equation*}
We claim that this collection of $\kk$-module isomorphisms over all finite sets $X$ are compatible with the $\FI$-module structures on $\ws_{-1}^\dag V$ and $SV$. 

To verify the claim, let $f:X\to Y$ be an injective map between finite sets, and let $\phi\in (\ws_{-1}^\dag V)_X$. Then one has:
\begin{align*}
\alpha_Y( f_*(\phi) ) &= f_*(\phi)_{Y\sqcup\{\star\}}((\eta_Y)_{Y\sqcup\{\star\}}(\id_{ Y\sqcup\{\star\} }))
= \phi_{Y\sqcup\{\star\}}( \ws_{-1}(\rho_f)_{Y\sqcup\{\star\}} ((\eta_Y)_{Y\sqcup\{\star\}}(\id_{ Y\sqcup\{\star\} })  ) ) \\
&= \phi_{Y\sqcup\{\star\}}(  (\eta_X)_{Y\sqcup\{\star\}}(f \sqcup \id_{\{\star\}} ) ) 
= \phi_{Y\sqcup\{\star\}}( (f\sqcup \id_{\{\star\}})_* ( (\eta_X)_{X\sqcup\{\star\}}(\id_{X\sqcup\{\star\}}) ) ) \\
&= (f\sqcup \id_{\{\star\}})_* (\phi_{X\sqcup\{\star\}}(  (\eta_X)_{X\sqcup\{\star\}}(\id_{X\sqcup\{\star\}}) )) 
= f_* (\alpha_X(\phi)).
\end{align*}
\end{proof}

\section{Right adjoint of the derivative functor}

Let $X$ be a finite set. It is known that there is an isomorphism
\begin{equation} \label{SM(X)}
\Theta_X: M(X) \oplus \left( \bigoplus_{x\in X} M(X\setminus \{x\}) \right) \longrightarrow SM(X) ,
\end{equation}
see \cite[Proof of Proposition 2.12]{CEFN}. We need this isomorphism explicitly. The restriction of $\Theta$ to the direct summand $M(X)$ is $\iota: M(X) \to SM(X)$. To define the restriction of $\Theta$ to the direct summand $M(X\setminus\{x\})$, we shall use the following notation.

\begin{notation}
Suppose $f: X\to Y$ is a map between finite sets. Suppose $\{w\}$ and $\{z\}$ are any one-element sets. We write $f\sqcup (w\to z)$ for the map $X\sqcup \{w\} \to Y\sqcup \{z\}$ whose restriction to $X$ is $f$ and which sends $w$ to $z$.
\end{notation}

Let $Y$ be a finite set. The map $\Theta_X : M(X\setminus\{x\})_Y \to SM(X)_Y$ is defined by
\begin{equation*}
f \mapsto f\sqcup (x\to \star)
\end{equation*}
for each injective map $f: X\setminus\{x\} \to Y$.  

\begin{theorem}
The functor $\ws_{-1}$ is a right adjoint of the derivative functor $D: \FI\Mod \to \FI\Mod$.
\end{theorem}

\begin{proof}
The functor $D$ is right exact and transforms direct sums to direct sums. By Proposition \ref{adjoint functor theorem}, we need to show that the functors $D^\dag$ and $\ws_{-1}$ are isomorphic.

Let $X$ be a finite set. It follows from above that we have an isomorphism
\begin{equation} \label{DM(X)}
\theta_X: \bigoplus_{x\in X} M(X\setminus \{x\}) \longrightarrow DM(X),
\end{equation}
where $\theta_X$ is defined by the composition of $\Theta_X$ with the quotient map $\pi_X: SM(X)\to DM(X)$. 

Let $V$ be an $\FI$-module. From Definition \ref{right adjoint} and (\ref{DM(X)}), we have a $\kk$-module isomorphism
\begin{equation*}
\beta_X: (D^\dag V)_X \longrightarrow (\ws_{-1}V)_X, \quad \phi\mapsto \sum_{x\in X} \phi_{X\setminus\{x\}}((\theta_X)_{X\setminus\{x\}}(\id_{X\setminus\{x\}})).
\end{equation*}
We claim that this collection of $\kk$-module isomorphisms over all finite sets $X$ are compatible with the $\FI$-module structures on $D^\dag V$ and $\ws_{-1}V$.

To verify the claim, let $f: X\to Y$ be an injective map between finite sets, and let $\phi\in(D^\dag V)_X$. Then one has:
\begin{align*}
\beta_Y (f_*(\phi)) &= \sum_{y\in Y}  f_*(\phi)_{Y\setminus\{y\}} ((\theta_Y)_{Y\setminus\{y\}} (\id_{Y\setminus\{y\}})) 
= \sum_{y\in Y} \phi_{Y\setminus\{y\}} ( D(\rho_f)_{Y\setminus\{y\}} (  (\theta_Y)_{Y\setminus\{y\}}( \id_{Y\setminus\{y\}}) ) )\\
&= \sum_{y\in Y} \phi_{Y\setminus\{y\}} ( \pi_X( (\id_{Y\setminus\{y\}} \sqcup (y\to \star) ) \circ f  )  )  
= \sum_{x\in X} \phi_{Y\setminus\{f(x)\}} ( (\theta_X)_{Y\setminus\{f(x)\}}( f\mid_{X\setminus\{x\}}   )  ) \\
&= \sum_{x\in X} \phi_{Y\setminus\{f(x)\}} ( (f\mid_{X\setminus\{x\}})_*((\theta_X)_{X\setminus\{x\}} (\id_{X\setminus\{x\}}) )  ) \\
&= \sum_{x\in X} (f\mid_{X\setminus\{x\}})_*( \phi_{X\setminus\{x\}} ((\theta_X)_{X\setminus\{x\}}(\id_{X\setminus\{x\}}) ) ) 
= f_*(\beta_X(\phi)).
\end{align*}
\end{proof}

\section{Coinduction functor}

The coinduction functor $Q$ on the category of $\FI$-modules was defined in \cite[Definition 4.1]{GL} as $S^\dag$. It is a right adjoint functor of $S$ by Proposition \ref{adjoint functor theorem} (see \cite[Lemma 4.2]{GL}). In this section, we give an explicit description of $Q$ in terms of $\ws_{-1}$.

\begin{notation}
Suppose $f: X \to Y$ is a map between finite sets. If $y\in Y\setminus f(X)$, then define 
\begin{equation*}
\partial_y f: X\to Y\setminus \{y\}
\end{equation*}
by $\partial_y f(x):=f(x)$ for each $x\in X$.
\end{notation}

\begin{notation}
Suppose $V$ is an $\FI$-module. If $f:X\to Y$ is a morphism in $\FI$, define the map
\begin{equation*}
\partial f_* : V_X \longrightarrow (\ws_{-1}V)_Y, \quad v \mapsto \sum_{y\in Y\setminus f(X)} (\partial_y f)_*(v),
\end{equation*}
where $(\partial_y f)_*: V_X \to V_{Y\setminus \{y\}}$ is defined by the $\FI$-module structure of $V$. 
\end{notation}

It is easily checked that if $f:X\to Y$ and $g:Y\to Z$ are morphisms in $\FI$, then one has:
\begin{equation} \label{leibniz rule for partial}
\partial(gf)_* = (\partial g_*) f_* + g_* (\partial f_*).
\end{equation}

We define an $\FI$-module $Q'V$ as follows. For each finite set $X$, let
\begin{equation} \label{Q'V}
(Q'V)_X := V_X \oplus (\ws_{-1}V)_X.
\end{equation}
If $f:X\to Y$ is a morphism in $\FI$, then define $f_*: (Q'V)_X \to (Q'V)_Y$ to be 
\begin{equation*}
\left(
\begin{array}{cc}
f_* & 0 \\
\partial f_* & f_*
\end{array}
\right),
\end{equation*}
where we use column notation for the direct sum in (\ref{Q'V}). It follows from (\ref{leibniz rule for partial}) that $Q'V$ is an $\FI$-module.

\begin{theorem} \label{right adjoint of S}
The coinduction functor $Q:\FI\Mod\to\FI\Mod$ is isomorphic to the functor $Q': V\mapsto Q'V$.
\end{theorem}

\begin{proof}
Let $V$ be an $\FI$-module. Let $X$ be any finite set. One has $QV = S^\dag V$ by definition of $Q$. From Definition \ref{right adjoint} and (\ref{SM(X)}), we have a $\kk$-module isomorphism
\begin{equation*}
\gamma_X : (S^\dag V)_X \longrightarrow (Q'V)_X, \quad \phi \mapsto \phi_X((\Theta_X)_X(\id_X)) + \sum_{x\in X} \phi_{X\setminus\{x\}} ( (\Theta_X)_{X\setminus\{x\}}  (\id_{X\setminus\{x\}})). 
\end{equation*}
We claim that this collection of $\kk$-module isomorphisms over all finite sets $X$ are compatible with the $\FI$-module structures on $S^\dag V$ and $Q'V$.

To verify the claim, let $f: X\to Y$ be an injective map between finite sets, and let $\phi\in(S^\dag V)_X$. Then one has:
\begin{align*}
\gamma_Y (f_*(\phi)) =&  f_*(\phi)_Y ((\Theta_Y)_Y(\id_Y)) + \sum_{y\in Y}  f_*(\phi)_{Y\setminus\{y\}} ((\Theta_Y)_{Y\setminus\{y\}}  (\id_{Y\setminus\{y\}})) \\
=& \phi_Y( S(\rho_f)_Y ( (\Theta_Y)_Y( \id_{Y}) ) ) + \sum_{y\in Y} \phi_{Y\setminus\{y\}} ( S(\rho_f)_{Y\setminus\{y\}} (  (\Theta_Y)_{Y\setminus\{y\}}  ( \id_{Y\setminus\{y\}}) ) )\\
=& \phi_Y( (\Theta_X)_Y(f) ) + \sum_{y\in Y} \phi_{Y\setminus\{y\}} (  (\id_{Y\setminus\{y\}} \sqcup (y\to \star) ) \circ f   )  \\
=& \phi_Y( (\Theta_X)_Y(f) ) 
+ \sum_{y\in Y\setminus f(X)} \phi_{Y\setminus\{y\}} ( (\Theta_X)_{Y\setminus\{y\}} (\partial_y f) ) \\
&+ \sum_{x\in X} \phi_{Y\setminus\{f(x)\}} ( (\Theta_X)_{Y\setminus\{f(x)\}} ( f\mid_{X\setminus\{x\}}   )  ) \\
=& \phi_Y(f_*( (\Theta_X)_X(\id_X))) 
+ \sum_{y\in Y\setminus f(X)} \phi_{Y\setminus\{y\}} ( (\partial_y f)_* ((\Theta_X)_X (\id_X)) ) \\
&+ \sum_{x\in X} \phi_{Y\setminus\{f(x)\}} ( (f\mid_{X\setminus\{x\}})_*((\Theta_X)_{X\setminus\{x\}}  (\id_{X\setminus\{x\}}) )  ) \\
=& f_*(\phi_X((\Theta_X)_X(\id_X))) 
+ \sum_{y\in Y\setminus f(X)}  (\partial_y f)_* ( \phi_X( (\Theta_X)_X (\id_X)) ) \\
&+ \sum_{x\in X} (f\mid_{X\setminus\{x\}})_*( \phi_{X\setminus\{x\}} ( (\Theta_X)_{X\setminus\{x\}}  (\id_{X\setminus\{x\}}) ) ) 
= f_*(\gamma_X(\phi)).
\end{align*}
\end{proof}

Let $\FB$ be the category whose objects are the finite sets and whose morphisms are the bijections. Then $\FB$ is a subcategory of $\FI$ and so there is a natural forgetful functor from $\FI\Mod$ to $\FB\Mod$ (see \cite{CE}). 

\begin{corollary} \label{QV as extension}
Let $V$ be an $\FI$-module. Then there is a short exact sequence
\begin{equation*}
0 \longrightarrow \ws_{-1}V \longrightarrow QV \longrightarrow V \longrightarrow 0
\end{equation*}
of $\FI$-modules which splits after applying the forgetful functor to the category of $\FB$-modules.
\end{corollary}
\begin{proof}
Immediate from Theorem \ref{right adjoint of S} and the definition of $Q'$.
\end{proof}

From Corollary \ref{QV as extension} and (\ref{wsM(X)}), we recover \cite[Theorem 1.3]{GL} for $\FI$.

\begin{remark}
In \cite{GL}, for any finite field $\mathbb{F}_q$, Gan and Li also studied the coinduction functor for $\mathrm{VI}$-modules where $\mathrm{VI}$ is the category whose objects are the finite dimensional vector spaces over $\mathbb{F}_q$ and whose morphisms are the injective linear maps. Similarly to $\FI$-modules, the coinduction functor for $\mathrm{VI}$-modules is defined as $S^\dag$ where $S$ is the shift functor for $\mathrm{VI}$-modules. However, we do not know of analogues of the negative-one shift functor $\ws_{-1}$ and the derivative functor $D$ for $\mathrm{VI}$-modules.
\end{remark}

\section*{Acknowledgments}
I am grateful to the referee for making detailed suggestions to improve the exposition of the paper.


\begin{thebibliography}{99}

\bibitem{CE} T. Church, J. Ellenberg, Homology of FI-modules, arXiv:1506.01022v1.

\bibitem{CEFN}  T. Church, J. Ellenberg, B. Farb, R. Nagpal, FI-modules over Noetherian rings. Geom. Top. 18-5 (2014), 2951-2984, arXiv:1210.1854v2.

\bibitem{GL} W.L. Gan, L. Li, Coinduction functor in representation stability theory. J. Lond. Math. Soc. (2) 92 (2015), no. 3, 689-711, arXiv:1502.06989v2.

\bibitem{Ga} W.L. Gan, A long exact sequence for homology of FI-modules, arXiv:1602.08873v2.

\bibitem{Li3} L. Li, Upper bounds of homological invariants of $\FI_G$-modules. Arch. Math. (Basel) 107 (2016), no. 3, 201-211, arXiv:1512.05879v2.

\bibitem{LRa} L. Li, E. Ramos, Depth and the Local Cohomology of $\FI_G$-modules, arXiv:1602.04405v2.

\bibitem{LYu} L. Li, N. Yu, Filtrations and Homological degrees of FI-modules, arXiv:1511.02977v2.

\bibitem{PN} J.F. Palmquist, D. Newell, Bifunctors and adjoint pairs. Trans. Amer. Math. Soc. 155 (1971) 293-303. 

\bibitem{Ra} E. Ramos, Homological Invariants of FI-modules and $\FI_G$-modules, arXiv:1511.03964v3.

\bibitem{Ro} J. Rotman, An introduction to homological algebra. Second edition. Universitext. Springer, New York, 2009. 


\end{thebibliography}
\end{document}